\documentclass[11pt]{article}
\usepackage[margin=1in]{geometry}
\usepackage{amsmath,amssymb,amsthm}
\usepackage{enumitem}
\usepackage{graphicx}
\usepackage[hidelinks]{hyperref}
\usepackage{microtype}

\newtheorem{theorem}{Theorem}
\newtheorem{lemma}[theorem]{Lemma}
\newtheorem{proposition}[theorem]{Proposition}
\newtheorem{corollary}[theorem]{Corollary}
\theoremstyle{definition}
\newtheorem{definition}[theorem]{Definition}
\theoremstyle{remark}
\newtheorem{remark}[theorem]{Remark}

\let\lg\relax
\DeclareMathOperator{\lg}{lg}
\DeclareMathOperator{\dc}{dc}
\DeclareMathOperator{\Lip}{Lip}

\newcommand{\Z}{\mathbb{Z}}
\newcommand{\Q}{\mathbb{Q}}
\newcommand{\R}{\mathbb{R}}
\newcommand{\norm}[1]{\lVert #1\rVert_\infty}
\newcommand{\CC}{\mathcal{C}}
\newcommand{\Dec}{\mathsf{Dec}}
\newcommand{\nInstOne}{260}
\newcommand{\maxRatioOne}{0.438}
\newcommand{\nInstTwo}{60}
\newcommand{\maxRatioTwo}{0.182}
\newcommand{\nInstThree}{20}
\newcommand{\maxRatioThree}{0.157}
\newcommand{\emaxOne}{2}
\newcommand{\eboundOne}{12}
\newcommand{\emaxTwo}{4}
\newcommand{\eboundTwo}{48}
\newcommand{\emaxThree}{16}
\newcommand{\eboundThree}{192}
\newcommand{\emaxFour}{48}
\newcommand{\eboundFour}{768}
\newcommand{\rampMax}{0.2785}
\newcommand{\rampArg}{1.278}
\newcommand{\rampWorst}{0.2785}
\newcommand{\bitsPerRatioMin}{8.1}

\newcommand{\maxUbFrac}{0.311}
\newcommand{\minLbRatio}{129.1}

\newcommand{\slopeOne}{0.90}
\newcommand{\slopePlainOne}{1.08}
\newcommand{\slopeTwo}{0.92}
\newcommand{\bitsPerATwo}{116.8}
\newcommand{\maxATwo}{631}
\newcommand{\slopeThree}{0.85}
\newcommand{\bitsPerAThree}{673.9}
\newcommand{\maxAThree}{364}
\newcommand{\plainPerRatioMax}{448}
\newcommand{\maxRatioL}{4\,000}

\title{Description Complexity Bounds for\\ Analytic Approximations of Lipschitz Functions}
\author{A.~A.~Vozmishchev\\[2pt]
\small Faculty of Applied Informatics, ITMO University, St.~Petersburg, 197101, Russia}
\date{October 2026}

\begin{document}
\maketitle

\begin{abstract}
We study how many bits are needed to write down an elementary analytic expression that uniformly approximates a Lipschitz function. Expressions are finite trees over the basis $\{+,-,\times,\div,\exp\}$ with the variables and rational constants, and their size is $\dc(f)$, the length of the shortest description of $f$ in a fixed binary description language. For functions of one variable we prove that every $L$-Lipschitz $F:[0,1]\to[0,1]$ has an everywhere-defined approximant of accuracy $\Delta$ with $\dc(f)\le C\,(L/\Delta+\lg(1/\Delta+2))$ for all $L,\Delta>0$, by a linear spline smoothed with logistic ramps in a description language with one repetition construct. We prove a matching lower bound valid for \emph{every} description language: for $\Delta\le 1/36$ and every $L>0$ some $L$-Lipschitz function needs at least $\tfrac1{16}(L/\Delta+\lg(1/\Delta+2))$ bits. Hence the worst-case description complexity is $\Theta(L/\Delta+\lg(1/\Delta+2))$. For functions on $[0,1]^d$ the same method gives $\Theta_d\bigl((L/\Delta)^d+\lg(1/\Delta+2)\bigr)$, using a tensor-product spline whose coefficients are mixed second differences and cost $O(d)$ bits each. Without the repetition construct the upper bounds carry an additional logarithmic factor, whose necessity for plain expression trees is left open. A reference encoder, decoder and evaluator confirm the constructions numerically, and figures and tables report the checks.
\end{abstract}


\section{Introduction}\label{sec:intro}

How complex must a closed-form expression be to represent a given continuous function to a prescribed accuracy? Classical results such as the Stone--Weierstrass theorem and Jackson's theorems~\cite{Jackson1930} show that polynomial approximants exist and give their degrees. They treat the number and precision of the coefficients as free. Here we charge for them. We measure an expression by the number of bits needed to describe it (Section~\ref{sec:framework}).

Fix $L,\Delta>0$ and let $F:[0,1]\to[0,1]$ be $L$-Lipschitz. What is the least description complexity of an expression $f$ over
\[
E=\{+,-,\times,\div,\exp,x,c\},\qquad c\in\Q,
\]
with $\norm{F-f}\le\Delta$? Write $\CC(L,\Delta)$ for the worst case over all such $F$ (Definition~\ref{def:CC}).

\paragraph{Contributions.} Logarithms $\lg$ are binary throughout.
\begin{enumerate}[label=(\roman*),leftmargin=*]
\item \emph{Upper bound} (Theorem~\ref{thm:upper}). For all $L,\Delta>0$, $\CC(L,\Delta)\le 435\,L/\Delta+12\lg(1/\Delta+2)$. The approximant is a linear spline smoothed by logistic ramps, and it is defined on all of $\mathbb R$.
\item \emph{Lower bound} (Theorem~\ref{thm:lower}, Proposition~\ref{prop:lowerlog}). For $\Delta\le 1/36$ and all $L>0$, $\CC(L,\Delta)\ge\frac1{16}\bigl(L/\Delta+\lg(1/\Delta+2)\bigr)$. This depends only on counting bit strings, so it holds for every description language.
\item \emph{Two-sided bound} (Corollary~\ref{cor:theta}). Together, $\CC(L,\Delta)=\Theta\bigl(L/\Delta+\lg(1/\Delta+2)\bigr)$. In particular the additive $\lg(1/\Delta)$ term is necessary.
\item \emph{Plain trees} (Proposition~\ref{prop:plain}). The upper bound uses a description language with one repetition construct (a loop over an arithmetic progression of breakpoints). Without it, the same construction costs $O((L/\Delta)\lg(L/\Delta+2)+\lg(1/\Delta+2))$. Whether the factor $\lg(L/\Delta)$ is necessary for plain trees is open (Section~\ref{sec:open}).
\item \emph{Several variables} (Theorem~\ref{thm:mainD}). For $F:[0,1]^d\to[0,1]$, $L$-Lipschitz in the sup norm, the worst-case complexity is $\Theta_d\bigl((L/\Delta)^d+\lg(1/\Delta+2)\bigr)$: the exponent $d$ is sharp and the constants grow exponentially in $d$.
\item \emph{Numerical verification} (Section~\ref{sec:numerics}). A reference encoder, independent decoder and evaluator for the description language confirm that the constructions decode to the claimed formulas, meet the error guarantees, and have the predicted description lengths.
\end{enumerate}
Both ingredients of the proofs are elementary: smoothed spline interpolation, and a packing/pigeonhole argument. The contribution is a precise complexity measure in which they match, together with an explicit account of how the answer depends on the description language.

\paragraph{Related work.} Kolmogorov and Tikhomirov determined the metric entropy of Lipschitz classes~\cite{KT1959}; see also~\cite{LGM1996}. Our lower bound is an elementary packing argument for that entropy that makes the admissible range of $\Delta$ explicit. The upper bound is linear-spline interpolation in which the ReLU ramp $\max\{0,z\}$ is replaced by the smooth ramp $z\,\sigma(z/\gamma)$, i.e.\ a scaled sigmoid-weighted linear unit (SiLU/``swish'')~\cite{EUD2018,RZL2017}. Approximation of Lipschitz and Sobolev classes by neural networks with explicit size bounds is well studied~\cite{Yarotsky2017}, and VC-dimension arguments give matching lower bounds on network size~\cite{BHLM2019}. Yarotsky and Zhevnerchuk~\cite{YZ2020} map out the trade-offs between the number of weights, depth, activation function and weight precision, and show that the total information $\epsilon^{-d/r}$ is split differently in different regimes; Yarotsky~\cite{Y2021} shows that fixed-size networks with certain elementary activations can approximate every continuous function by exploiting unbounded weight precision, while standard activations cannot (by Khovanskii's bounds~\cite{Khovanskii1991}). Those results count weights or neurons; we count bits. Bounds in terms of the number of nodes of a formula with $\exp$ would need pseudo-dimension bounds for Pfaffian formulas~\cite{KM1997}; we make no claim in that direction (Remark~\ref{rem:nodes}). Computability of real functions is treated in~\cite{Weihrauch2000}, and description length and Kolmogorov complexity in~\cite{LV2019}. Integers are encoded with Elias' $\gamma$-code~\cite{Elias1975}.

\section{Formal Framework}\label{sec:framework}

\paragraph{Notation.} $\lg=\log_2$; $\exp$ is the natural exponential; $\lceil\cdot\rceil$ is the ceiling and $\lfloor\cdot\rceil$ the nearest integer. For $L>0$, $\Lip_L$ is the class of $F:[0,1]\to[0,1]$ with $|F(x)-F(y)|\le L|x-y|$. For $\Delta>0$ put $\Lambda_\Delta:=\lg(1/\Delta+2)$.

\subsection{Formulas}

\begin{definition}[Formulas]\label{def:formula}
A \emph{formula} is a finite ordered rooted tree whose internal nodes are labelled $+,-,\times,\div$ (two children) or $\exp$ (one child), and whose leaves are labelled by the variable $x$ or by a constant $r\in\Q$. It denotes a partial function of $x\in\R$ in the obvious way. There is no unary minus; $-t$ is written $0-t$. A formula is \emph{admissible} if every $\div$-node has a nonzero denominator at every $x\in[0,1]$. Admissible formulas denote continuous functions $[0,1]\to\R$. For $F:[0,1]\to\R$ and a formula $f$ we set $\norm{F-f}:=\sup_{x\in[0,1]}|F(x)-f(x)|$ if $f$ is admissible and $\norm{F-f}:=+\infty$ otherwise.
\end{definition}

\subsection{Descriptions and description complexity}

The measure $\dc$ is defined through an explicit decoder. The lower bounds use only that the decoder is a function on bit strings.

\begin{definition}[Description language]\label{def:language}
\emph{Symbols.} There are ten node symbols: $+,-,\times,\div,\exp,x,c,B,\iota,\kappa$. Each is written with a fixed distinct $4$-bit code (six of the sixteen codes are unused).

\emph{Integers.} For $n\ge1$, $\gamma(n)$ is Elias' $\gamma$-code: $\lfloor\lg n\rfloor$ zeros followed by the binary expansion of $n$. It has length $2\lfloor\lg n\rfloor+1\le 2\lg n+1$ and is prefix-free.

\emph{Rationals.} A rational $r=p/q'$ ($p\in\Z$, $q'\ge1$, not necessarily reduced) is written as one sign bit for $p$, then $\gamma(|p|+1)$, then $\gamma(q')$. Its length is
\[
\lambda(p/q')\le 3+2\lg(|p|+1)+2\lg q'.
\]
\emph{Descriptions.} Descriptions are bit strings generated by the following productions, each written in prefix order.
\begin{enumerate}[label=(D\arabic*),leftmargin=*,itemsep=1pt]
\item the symbol $x$;
\item the symbol $c$ followed by a rational;
\item the symbol $\exp$ followed by a description;
\item a symbol $\circ\in\{+,-,\times,\div\}$ followed by two descriptions;
\item (\emph{block}) the symbol $B$, then $\gamma(n)$ with $n\ge1$, two rationals $t_0$ (offset) and $\tau$ (step), then $\gamma(b)$ with $b\ge1$, then $n$ bit strings $w_1,\dots,w_n$ of length $b$ each, then a \emph{template} $T$. A template is a description generated by (D1)--(D4) in which the leaf symbols $\iota$ and $\kappa$ (with no payload) may additionally occur. Blocks may not be nested, and $\iota,\kappa$ may not occur outside a template.
\end{enumerate}
\emph{Decoder.} $\Dec(w)$ is the formula described by $w$, if $w$ is exactly one description, and is undefined otherwise. For (D1)--(D4) the meaning is the obvious tree. For a block, let $c_j:=(\text{the integer with binary expansion }w_j)-2^{b-1}$, so $c_j\in[-2^{b-1},2^{b-1}-1]$. Let $T_j$ be $T$ with every $\iota$ replaced by the constant $t_0+(j-1)\tau$ and every $\kappa$ by the constant $c_j$. The block denotes the formula $T_1+(T_2+(\cdots+(T_{n-1}+T_n)))$, which is $T_1$ if $n=1$.
\end{definition}

The first symbol of a description determines the production, rationals and integers are self-delimiting, and the $w_j$ have known length. Hence a bit string parses in at most one way, and $\Dec$ is a well-defined partial function.

\begin{definition}[Description complexity]\label{def:dc}
For a formula $f$, $\dc(f):=\min\{|w|:\Dec(w)=f\}$, where $|w|$ is the bit length of $w$. Every formula has a description that uses no block, so $\dc(f)<\infty$. We write $\dc_{\mathrm{plain}}(f)$ for the same quantity when (D5) is removed. Clearly $\dc\le\dc_{\mathrm{plain}}$.
\end{definition}

\begin{definition}[Worst-case complexity]\label{def:CC}
For $L,\Delta>0$ let
\[
\CC(L,\Delta):=\sup_{F\in\Lip_L}\;\min\{\dc(f):\ \norm{F-f}\le\Delta\}.
\]
The inner minimum exists because $\dc$ is integer valued, and the set is nonempty by Theorem~\ref{thm:upper}. $\CC_{\mathrm{plain}}$ is defined likewise with $\dc_{\mathrm{plain}}$.
\end{definition}

\begin{remark}[Dependence on the language]\label{rem:language}
$\dc$ depends on $\Dec$. The block (D5) is a deliberately minimal repetition construct: a loop whose index runs over an arithmetic progression and whose coefficients come from a short list. Without it, a regular family of breakpoints $i/N$ costs $\Theta(\lg N)$ bits per breakpoint. The upper bounds in Section~\ref{sec:upper} are proved for the language of Definition~\ref{def:language}. The lower bounds in Section~\ref{sec:lower} hold for \emph{any} partial map $\Dec$ from bit strings to formulas, since they use only the number of strings of a given length.
\end{remark}

\begin{remark}[Node complexity]\label{rem:nodes}
One may also count the nodes $|f|$ of a formula, charging one unit per constant regardless of its precision. This is a different measure. Constants of unbounded bit-length make the counting argument of Section~\ref{sec:lower} inapplicable, and lower bounds in terms of $|f|$ would need finer tools such as pseudo-dimension bounds for formulas with $\exp$~\cite{Khovanskii1991,KM1997}. We make no claim about $|f|$.
\end{remark}

\section{Upper Bound}\label{sec:upper}

\begin{lemma}[Smooth ramp]\label{lem:ramp}
For $\gamma>0$ let $R_\gamma(z):=\dfrac{z}{1+e^{-z/\gamma}}$ for $z\in\R$. Then
\[
\bigl|\max\{0,z\}-R_\gamma(z)\bigr|\le\frac{\gamma}{3}\qquad\text{for all }z\in\R.
\]
\end{lemma}
\begin{proof}
Put $u=z/\gamma$ and $\varphi(v)=v/(1+e^{v})$ for $v\ge0$. If $z\ge0$, then
\[
z-R_\gamma(z)=z\,\frac{e^{-u}}{1+e^{-u}}=\frac{z}{1+e^{u}}=\gamma\varphi(u).
\]
If $z<0$, then $-u>0$ and $|0-R_\gamma(z)|=|z|/(1+e^{-u})=\gamma\varphi(|u|)$. Since $e^v\ge1+v+v^2/2$ for $v\ge0$,
\[
\varphi(v)\le\frac{v}{2+v+v^2/2}\le\frac13,
\]
where the last inequality is equivalent to $\tfrac12(v-2)^2\ge0$.
\end{proof}

\begin{theorem}[Upper bound]\label{thm:upper}
Let $L,\Delta>0$ and $F\in\Lip_L$. There is an admissible formula $f$ with $\norm{F-f}\le\Delta$ and
\[
\dc(f)\le 435\,\frac{L}{\Delta}+12\,\Lambda_\Delta,\qquad \Lambda_\Delta=\lg\Bigl(\frac1\Delta+2\Bigr).
\]
Consequently $\CC(L,\Delta)\le 435\,L/\Delta+12\,\Lambda_\Delta$. The formula depends on $F$ only through the values $F(i/N)$ at the grid points of the proof.
\end{theorem}

\begin{proof}
Put $q:=\lceil2/\Delta\rceil$. All rounding is to multiples of $1/q$.

\emph{Case 1: $\Delta\ge1/2$.} The constant $1/2$ satisfies $|F-\tfrac12|\le\tfrac12\le\Delta$. Its description is the symbol $c$ followed by $\lambda(1/2)=1+3+3=7$ bits, so $\dc=11\le 11\Lambda_\Delta$ because $\Lambda_\Delta>1$.

\emph{Case 2: $\Delta<1/2$ and $L\le\Delta/2$.} Since $F$ is $L$-Lipschitz on an interval of length $1$, $\max F-\min F\le L$. Let $m=\tfrac12(\max F+\min F)\in[0,1]$ and $c=\lfloor qm\rceil/q=p/q$ with $0\le p\le q$. Then
\[
\norm{F-c}\le \frac L2+\frac1{2q}\le\frac\Delta4+\frac\Delta4<\Delta.
\]
The description has length $4+\lambda(p/q)\le 7+4\lg(q+1)$. As $\Delta<\frac12$, $q<\frac2\Delta+1<\frac5{2\Delta}$, so $q+1\le 3/\Delta\le(1/\Delta+2)^2$ and $\lg(q+1)\le2\Lambda_\Delta$. Also $\Lambda_\Delta>2$. So $\dc\le 7+8\Lambda_\Delta\le12\Lambda_\Delta$.

\emph{Case 3: $\Delta<1/2$ and $L>\Delta/2$.} Then $qL\ge 2L/\Delta>1$. Set
\[
N:=\lceil qL\rceil\ \ (\ge2),\qquad x_i:=\frac iN\ (0\le i\le N).
\]
Then $N\ge qL$. Since $qL>1$ and $q<5/(2\Delta)$, also $N<qL+1<2qL<5L/\Delta$.

\emph{Step 1 (rounded node values).} Let $p_i:=\lfloor qF(x_i)\rceil\in\{0,\dots,q\}$ and $y_i:=p_i/q$. Then $|y_i-F(x_i)|\le\frac1{2q}\le\frac\Delta4$.

\emph{Step 2 (spline error).} Let $S$ be the continuous piecewise linear function with $S(x_i)=y_i$, and $S_F$ the one with $S_F(x_i)=F(x_i)$. For $x\in[x_i,x_{i+1}]$ write $t=(x-x_i)N\in[0,1]$. Then $F(x)-S_F(x)=(1-t)(F(x)-F(x_i))+t(F(x)-F(x_{i+1}))$, so
\[
|F(x)-S_F(x)|\le (1-t)\frac{Lt}{N}+t\frac{L(1-t)}{N}=\frac{2Lt(1-t)}N\le\frac{L}{2N}\le\frac1{2q}\le\frac\Delta4,
\]
using $N\ge qL$. On each cell $S-S_F$ is a convex combination of the node errors, so $\norm{S-S_F}\le\frac\Delta4$. Hence $\norm{F-S}\le\frac\Delta2$.

\emph{Step 3 (integer slopes).} Let $u_i:=p_{i+1}-p_i\in\Z$ for $0\le i\le N-1$. The slope of $S$ on $[x_i,x_{i+1}]$ is $s_i=N(y_{i+1}-y_i)=\frac Nq u_i$. Since $|y_{i+1}-y_i|\le|F(x_{i+1})-F(x_i)|+\frac1q\le\frac LN+\frac1q$,
\[
|u_i|\le \frac{qL}{N}+1\le2,\qquad\text{so } u_i\in\{-2,\dots,2\}.
\]
For $1\le i\le N-1$ the slope changes are $m_i:=s_i-s_{i-1}=\frac Nq e_i$ with integers $e_i:=u_i-u_{i-1}\in\{-4,\dots,4\}$. On each cell $S$ has slope $s_j=s_0+\sum_{i\le j}m_i$ and $S(0)=y_0$, so
\begin{equation}\label{eq:S}
S(x)=y_0+s_0x+\sum_{i=1}^{N-1}m_i\max\{0,x-x_i\}.
\end{equation}
Indeed the right side is continuous, piecewise linear, has breakpoints in $\{x_i\}$, takes the value $y_0$ at $0$, and has slope $s_0+\sum_{i\le j}m_i=s_j$ on $(x_j,x_{j+1})$.

\emph{Step 4 (smoothing).} Let $k:=2\lceil\lg N\rceil+2$ and $\gamma:=2^{-k}$. Then $\gamma\le1/(4N^2)$. Define
\begin{equation}\label{eq:f}
f(x):=\frac{p_0}q+\frac Nq\Bigl(u_0\,x+\sum_{i=1}^{N-1}e_i\,R_\gamma\bigl(x-\tfrac iN\bigr)\Bigr).
\end{equation}
This replaces $\max\{0,\cdot\}$ in~\eqref{eq:S} by $R_\gamma$. By Lemma~\ref{lem:ramp} and $|m_i|\le 4N/q$,
\[
\norm{S-f}\le\sum_{i=1}^{N-1}|m_i|\frac\gamma3<\frac{4N^2\gamma}{3q}\le\frac1{3q}\le\frac\Delta6 .
\]

\emph{Step 5 (total error).} $\norm{F-f}\le\frac\Delta2+\frac\Delta6<\Delta$.

\emph{Step 6 (the formula over $E$, admissibility).} Since $1/\gamma=2^k$, we have $R_\gamma(z)=z\div\bigl(1+\exp(0-(z\times 2^k))\bigr)$. Each $\div$ has denominator $\ge1$ and the $\exp$ argument is bounded on $[0,1]$, so $f$ is admissible (indeed defined on all of $\R$).

\emph{Step 7 (description length).} Describe $f$ as
\[
+\ \ c(\tfrac{p_0}q)\ \ \times\ \ c(\tfrac Nq)\ \ +\ \ \times\ \ c(u_0)\ \ x\ \ B(n{=}N{-}1;\ t_0{=}\tau{=}\tfrac1N;\ b{=}4;\ w_j{=}e_j{+}8;\ T),
\]
where the template is
\[
T=\ \kappa\times\Bigl((x-\iota)\div\bigl(1+\exp\bigl(0-((x-\iota)\times 2^k)\bigr)\bigr)\Bigr).
\]
Here $b=4$ gives the range $[-8,7]\supset[-4,4]$, and $T_j=e_j\,R_\gamma(x-j/N)$ since $t_0+(j-1)\tau=j/N$, so the block expands to the sum in~\eqref{eq:f}. The costs are as follows.
\begin{itemize}[itemsep=1pt]
\item The 9 symbols outside the block header and template ($+,c,\times,c,+,\times,c,x,B$) cost $36$ bits.
\item Constants: $\lambda(p_0/q)\le3+4\lg(q+1)$ since $p_0\le q$; $\lambda(N/q)\le3+2\lg(N+1)+2\lg q$; $\lambda(u_0)\le5$.
\item Block header: $\gamma(N-1)\le2\lg N+1$; $\lambda(a)=\lambda(d)\le5+2\lg N$ each; $\gamma(4)$ has $5$ bits.
\item Coefficients: $4(N-1)$ bits.
\item Template: $16$ symbols ($64$ bits), plus constants $1$, $0$ and $2^k$ costing $5$, $3$ and $2k+3$ bits. Since $k\le 2\lg N+4$, the total is $\le 83+4\lg N$.
\end{itemize}
Summing, and using $2\lg(N+1)+10\lg N\le12\lg(N+1)$ and $4\lg(q+1)+2\lg q\le6\lg(q+1)$,
\[
\dc(f)\le 4N+12\lg(N+1)+6\lg(q+1)+142 .
\]
Now $\lg(N+1)\le N$ and $142\le 71N$ (as $N\ge2$), so the first three terms are at most $87N<435\,L/\Delta$. Also $6\lg(q+1)\le12\Lambda_\Delta$ as in Case~2. This gives the claim in Case 3, and Cases 1 and 2 are bounded by $12\Lambda_\Delta$.
\end{proof}

\begin{proposition}[Plain trees]\label{prop:plain}
Under the same hypotheses, the formula $f$ of Theorem~\ref{thm:upper} satisfies
\[
\dc_{\mathrm{plain}}(f)=O\!\Bigl(\frac L\Delta\lg\Bigl(\frac L\Delta+2\Bigr)+\Lambda_\Delta\Bigr).
\]
\end{proposition}
\begin{proof}
Cases 1 and 2 are unchanged. In Case 3, write out the $N-1$ copies $T_j$ explicitly. Each costs $64$ bits for symbols, $2(3+4\lg N)$ bits for the two copies of the breakpoint $j/N$ (each $\lambda(j/N)\le3+4\lg N$), at most $7$ bits for $e_j$, $5+3$ bits for the constants $1,0$, and $2k+3\le 4\lg N+11$ bits for $2^k$. That is at most $96+12\lg N\le 108\lg N$ bits, using $N\ge2$. Joining the terms costs $4$ bits per $+$. The remaining parts cost $O(\lg N+\lg q)$ bits. Since $N<5L/\Delta$ and $q+1\le(1/\Delta+2)^2$, the total is $O(N\lg N+\Lambda_\Delta)$, which gives the claim.
\end{proof}

\section{Lower Bound}\label{sec:lower}

\begin{lemma}[Counting]\label{lem:count}
For every $M\ge0$, at most $2^{M+1}-1$ formulas $f$ satisfy $\dc(f)\le M$. The same holds for any partial map $\Dec$ from bit strings to formulas.
\end{lemma}
\begin{proof}
Each such $f$ equals $\Dec(w)$ for some bit string $w$ of length at most $M$, and there are $2^{M+1}-1$ such strings.
\end{proof}

\begin{lemma}[Separated families]\label{lem:pigeon}
Let $A\subseteq\Lip_L$ with $\norm{F-F'}>2\Delta$ for all distinct $F,F'\in A$, and suppose $|A|\ge2^n$ for an integer $n\ge1$. Then some $F\in A$ satisfies $\norm{F-f}>\Delta$ for every formula $f$ with $\dc(f)\le n-1$.
\end{lemma}
\begin{proof}
Suppose every $F\in A$ had a formula $f_F$ with $\dc(f_F)\le n-1$ and $\norm{F-f_F}\le\Delta$. If $f_F=f_{F'}$, the triangle inequality gives $\norm{F-F'}\le2\Delta$, so $F=F'$. Thus $F\mapsto f_F$ is injective, and $|A|\le 2^{n}-1$ by Lemma~\ref{lem:count}. This contradicts $|A|\ge2^n$.
\end{proof}

\begin{lemma}[Bump family]\label{lem:bump}
Let $0<\Delta<\frac12$ and $L\ge8\Delta$, and put $n:=\lceil L/(4\Delta)\rceil-1$. Then $n\ge L/(8\Delta)\ge 1$, and there is a family $\mathcal G\subseteq\Lip_L$ with $|\mathcal G|=2^n$ and $\norm{G-G'}>2\Delta$ for distinct members.
\end{lemma}
\begin{proof}
Since $L/(4\Delta)\ge2$ and $1\le L/(8\Delta)$, we get $n\ge L/(4\Delta)-1\ge L/(8\Delta)\ge1$. Moreover $n<L/(4\Delta)$ by definition. Let $h:=\min\{1,L/(2n)\}$. Then $h>2\Delta$, because $2\Delta<1$ and $L/(2n)>2\Delta$. Let $c_j=(2j-1)/(2n)$ and $\phi_j(x)=\max\{0,h-L|x-c_j|\}$ for $j=1,\dots,n$. Each $\phi_j$ has support of half-width $h/L\le 1/(2n)$, so the supports lie in the cells $[\frac{j-1}n,\frac jn]$ and overlap at most at endpoints. For $\omega\in\{0,1\}^n$ put $G_\omega=\sum_j\omega_j\phi_j$. Then $G_\omega$ is continuous, piecewise linear with slopes in $\{0,\pm L\}$, and takes values in $[0,h]\subseteq[0,1]$. So $G_\omega\in\Lip_L$. For $\omega\ne\omega'$ the functions differ at some $c_j$ by exactly $h$, and nowhere by more than $h$, so $\norm{G_\omega-G_{\omega'}}=h>2\Delta$.
\end{proof}

\begin{theorem}[Lower bound]\label{thm:lower}
Let $0<\Delta<\frac12$ and $L\ge8\Delta$. There is $F\in\Lip_L$ such that $\norm{F-f}>\Delta$ for every formula $f$ with $\dc(f)<L/(8\Delta)$. Hence $\CC(L,\Delta)\ge L/(8\Delta)$. This holds for every description language.
\end{theorem}
\begin{proof}
Apply Lemma~\ref{lem:pigeon} to the family of Lemma~\ref{lem:bump}, with $n\ge L/(8\Delta)$. Every $f$ with $\dc(f)<L/(8\Delta)\le n$ has $\dc(f)\le n-1$, since $\dc$ is an integer. The function $F$ is chosen once for all such $f$.
\end{proof}

\begin{proposition}[The additive term is necessary]\label{prop:lowerlog}
Let $0<\Delta\le1/36$ and $L>0$. Then $\CC(L,\Delta)\ge\frac12\lg(1/\Delta)$. This holds for every description language.
\end{proposition}
\begin{proof}
Let $K=\lfloor1/(3\Delta)\rfloor+1\ge 12$ and consider the constants $j\cdot3\Delta\in[0,1]$, $0\le j<K$. Each is $L$-Lipschitz, and distinct ones differ by at least $3\Delta>2\Delta$. Apply Lemma~\ref{lem:pigeon} with $n=\lfloor\lg K\rfloor\ge1$ (so $K\ge2^n$): some constant has all formulas of $\dc\le n-1$ at distance $>\Delta$, so $\CC\ge n$. Finally
\[
n\ge\lg K-1\ge\lg\frac1{3\Delta}-1=\lg\frac1\Delta-\lg6\ge\frac12\lg\frac1\Delta,
\]
the last step because $\Delta\le1/36$.
\end{proof}

\begin{corollary}[Two-sided bound]\label{cor:theta}
For $0<\Delta\le1/36$ and all $L>0$,
\[
\frac1{16}\Bigl(\frac L\Delta+\Lambda_\Delta\Bigr)\ \le\ \CC(L,\Delta)\ \le\ 435\,\frac L\Delta+12\,\Lambda_\Delta ,
\]
that is, $\CC(L,\Delta)=\Theta\bigl(L/\Delta+\lg(1/\Delta+2)\bigr)$.
\end{corollary}
\begin{proof}
The upper bound is Theorem~\ref{thm:upper}. For the lower bound let $A=L/\Delta$ and $B=\lg(1/\Delta)>5$. Since $1/\Delta+2\le1/\Delta^2$, we have $\Lambda_\Delta\le2B$. If $A\ge8$, Theorem~\ref{thm:lower} and Proposition~\ref{prop:lowerlog} give $\CC\ge\max\{A/8,B/2\}\ge A/16+B/4\ge(A+\Lambda_\Delta)/16$. If $A<8$, then $A+\Lambda_\Delta<8+2B\le4B$, and Proposition~\ref{prop:lowerlog} gives $\CC\ge B/2\ge(A+\Lambda_\Delta)/8$.
\end{proof}

\begin{remark}[Relation to metric entropy]
The lower bounds say that the class $\Lip_L$ has at least $2^{\Omega(L/\Delta)}$ elements at mutual distance $>2\Delta$, while $2^{M+1}$ strings can name at most $2^{M+1}$ formulas. This is the packing form of the Kolmogorov--Tikhomirov estimate $\lg N(\Lip_L,\Delta)\asymp L/\Delta$~\cite{KT1959}. Proving the packing directly shows where the hypotheses $\Delta<\frac12$ and $L\gtrsim\Delta$ are needed. For $\Delta\ge\frac12$ the single constant $\frac12$ approximates every $F:[0,1]\to[0,1]$, so no lower bound can hold there.
\end{remark}

\section{The Multidimensional Case}\label{sec:multid}

In this section $d\ge1$ is fixed. We extend the results of Sections~\ref{sec:upper} and~\ref{sec:lower} to Lipschitz functions of $d$ variables.

\paragraph{Class and complexity.} $\Lip_L^{(d)}$ is the class of $F:[0,1]^d\to[0,1]$ with $|F(x)-F(y)|\le L\lVert x-y\rVert_\infty$. Formulas are as in Definition~\ref{def:formula}, but with leaves $x_1,\dots,x_d$ or rational constants. A formula is admissible if no $\div$-node has a zero denominator on $[0,1]^d$, and $\norm{F-f}:=\sup_{[0,1]^d}|F-f|$ for admissible $f$ and $+\infty$ otherwise. The description complexity $\dc(f)$ is the length of the shortest string decoding to $f$ in the language below. We put
\[
\CC_d(L,\Delta):=\sup_{F\in\Lip_L^{(d)}}\ \min\{\dc(f):\ \norm{F-f}\le\Delta\},\qquad
\Lambda:=\lg\Bigl(\frac1\Delta+2\Bigr),\qquad A:=\Bigl(\frac L\Delta\Bigr)^d .
\]

\paragraph{The language $\mathcal L_d$.} This is Definition~\ref{def:language} with two changes.
\begin{itemize}[itemsep=2pt]
\item The leaves $x_k$ and $\iota_k$ are written as the symbol followed by $\gamma(k)$, where $\gamma$ is Elias' $\gamma$-code, so each costs at most $4+2\lg d+1$ bits. The symbols are $+,-,\times,\div,\exp,c,B,x,\iota,\kappa$, each with a $4$-bit code. A rational $p/q'$ is written as a sign bit, $\gamma(|p|+1)$ and $\gamma(q')$, and its length satisfies $\lambda(p/q')\le3+2\lg(|p|+1)+2\lg q'$.
\item The block (D5) has $r$ nested loops. It consists of the symbol $B$, then $\gamma(r)$ with $1\le r\le d$, then $\gamma(n)$, two rationals $t_0$ (offset) and $\tau$ (step), then $\gamma(b)$, then $n^r$ words $w_j\in\{0,1\}^b$ indexed by $j\in[n]^r$ in lexicographic order, then a template $T$. The template is built from the productions (D1)--(D4) and may also use the leaves $\iota_1,\dots,\iota_r$ and $\kappa$. The block denotes the right-nested sum of the formulas $T_j$ over $j\in[n]^r$ in lexicographic order. In $T_j$, each $\iota_m$ is replaced by the constant $t_0+(j_m-1)\tau$ and $\kappa$ by the constant $(\text{integer value of }w_j)-2^{b-1}$.
\end{itemize}
The lower bounds below hold for any decoder, and only the upper bound uses $\mathcal L_d$.

\begin{theorem}\label{thm:mainD}
There are constants $\alpha_d,\beta_d$ depending only on $d$ such that for all $L>0$ and all $0<\Delta\le1/36$,
\[
\frac{A+\Lambda}{2\cdot8^d}\ \le\ \CC_d(L,\Delta)\ \le\ \alpha_d\,A+\beta_d\,\Lambda .
\]
The upper bound holds for all $\Delta>0$. One may take $\alpha_d=20^d(4d^2+4d+7)$ and $\beta_d=\max\{11,\,B'_d+A'_d/2\}$ with $A'_d,B'_d$ as in Section~\ref{sec:bitsD}.
\end{theorem}

\subsection{Lower bound for $d\ge1$}\label{sec:lowerD}

\begin{lemma}[Separated families]\label{lem:sepD}
Let $\mathcal A\subseteq\Lip_L^{(d)}$ satisfy $\norm{F-F'}>2\Delta$ for all distinct $F,F'\in\mathcal A$, and let $|\mathcal A|\ge2^n$ for an integer $n\ge1$. Then some $F\in\mathcal A$ has $\norm{F-f}>\Delta$ for every formula $f$ with $\dc(f)\le n-1$. In particular $\CC_d(L,\Delta)\ge n$. This holds for every decoder.
\end{lemma}
\begin{proof}
At most $2^{n}-1$ bit strings have length $\le n-1$, so at most $2^n-1$ formulas have $\dc\le n-1$. If every $F\in\mathcal A$ had a formula $f_F$ with $\dc(f_F)\le n-1$ and $\norm{F-f_F}\le\Delta$, then $f_F=f_{F'}$ would give $\norm{F-F'}\le2\Delta$, hence $F=F'$. So $F\mapsto f_F$ would be injective and $|\mathcal A|\le2^n-1$, a contradiction.
\end{proof}

\begin{lemma}[Bump family]\label{lem:bumpD}
Let $0<\Delta<\frac12$ and $L\ge8\Delta$, and put $m:=\lceil L/(4\Delta)\rceil-1$. Then $m\ge L/(8\Delta)\ge1$, and there is a family $\mathcal G\subseteq\Lip_L^{(d)}$ with $|\mathcal G|=2^{m^d}$ and $\norm{G-G'}>2\Delta$ for distinct members.
\end{lemma}
\begin{proof}
By definition $m<L/(4\Delta)$. Since $L\ge8\Delta$ we have $L/(4\Delta)\ge2$ and $1\le L/(8\Delta)$, so $m\ge L/(4\Delta)-1\ge L/(8\Delta)\ge1$.
Let $\eta:=\min\{1,L/(2m)\}$. Then $\eta>2\Delta$, since $2\Delta<1$ and $L/(2m)>2\Delta$ (because $m<L/(4\Delta)$).
For $j\in[m]^d$ let $c_j=\bigl(\frac{2j_1-1}{2m},\dots,\frac{2j_d-1}{2m}\bigr)$ and $\psi_j(x)=\max\{0,\eta-L\lVert x-c_j\rVert_\infty\}$.
Each $\psi_j$ is $L$-Lipschitz for $\lVert\cdot\rVert_\infty$, takes values in $[0,\eta]\subseteq[0,1]$, has $\psi_j(c_j)=\eta$, and is supported in the open $\ell_\infty$-ball of radius $\eta/L\le1/(2m)$ about $c_j$. Distinct centres are at $\ell_\infty$-distance $\ge1/m$, so these balls are pairwise disjoint.
For $\omega\in\{0,1\}^{[m]^d}$ put $G_\omega:=\sum_j\omega_j\psi_j=\max_j\omega_j\psi_j$. A maximum of $L$-Lipschitz functions is $L$-Lipschitz, so $G_\omega\in\Lip_L^{(d)}$. If $\omega\ne\omega'$ differ at $j$, then $|G_\omega(c_j)-G_{\omega'}(c_j)|=\eta$, because the other bumps vanish at $c_j$. Everywhere the difference is at most $\eta$. Hence $\norm{G_\omega-G_{\omega'}}=\eta>2\Delta$.
\end{proof}

\begin{proposition}[Main term]\label{prop:lowerAD}
If $0<\Delta<\frac12$ and $L\ge8\Delta$, then $\CC_d(L,\Delta)\ge m^d\ge\bigl(L/(8\Delta)\bigr)^d$.
\end{proposition}
\begin{proof}
Apply Lemma~\ref{lem:sepD} to the family of Lemma~\ref{lem:bumpD} with $n=m^d$.
\end{proof}

\begin{proposition}[Additive term]\label{prop:lowerBD}
If $0<\Delta\le1/36$ and $L>0$, then $\CC_d(L,\Delta)\ge\frac12\lg(1/\Delta)$.
\end{proposition}
\begin{proof}
Let $K=\lfloor1/(3\Delta)\rfloor+1\ge12$. The constant functions $j\cdot3\Delta\in[0,1]$, $0\le j<K$, lie in $\Lip_L^{(d)}$ and differ pairwise by $\ge3\Delta>2\Delta$. Lemma~\ref{lem:sepD} with $n=\lfloor\lg K\rfloor\ge1$ gives
\[
\CC_d\ge n\ge\lg K-1\ge\lg\frac1{3\Delta}-1=\lg\frac1\Delta-\lg6\ge\frac12\lg\frac1\Delta,
\]
the last step because $\Delta\le1/36$.
\end{proof}

\begin{proof}[Proof of the lower bound in Theorem~\ref{thm:mainD}]
Let $B=\lg(1/\Delta)\ge\lg36>5$. Since $\Delta\le\frac12$ we have $\frac1\Delta+2\le\frac1{\Delta^2}$, hence $\Lambda\le2B$.

\emph{Case $L\ge8\Delta$.} Then $A\ge8^d$. By Propositions~\ref{prop:lowerAD} and~\ref{prop:lowerBD},
\[
\CC_d\ge\max\Bigl\{\frac A{8^d},\frac B2\Bigr\}\ge\frac12\Bigl(\frac A{8^d}+\frac\Lambda4\Bigr)\ge\frac{A+\Lambda}{2\cdot8^d},
\]
using $B/2\ge\Lambda/4$ and $\frac14\ge8^{-d}$.

\emph{Case $L<8\Delta$.} Then $A<8^d$, and $A+\Lambda<8^d+2B\le8^dB$, because $B>5$ gives $8^d/B+2\le8^d/5+2\le8^d$. By Proposition~\ref{prop:lowerBD}, $\CC_d\ge B/2>(A+\Lambda)/(2\cdot8^d)$.
\end{proof}

\subsection{Upper bound for $d\ge1$}\label{sec:upperD}

\subsubsection{The obstacle for $d\ge2$}

The one-dimensional proof of Theorem~\ref{thm:upper} stores integer slope changes in $\{-4,\dots,4\}$, which cost $O(1)$ bits each. A naive $d$-dimensional version (multilinear interpolation with coefficients $y_i=p_i/q$) would store every grid value $p_i\in\{0,\dots,q\}$. That costs $\lg q\approx\lg(1/\Delta)$ bits per node, giving $N^d\lg(1/\Delta)$ in total, which is too much. The remedy is to use \emph{mixed differences} in every coordinate. Their weights sum to zero, so the Lipschitz condition makes them bounded integers.

\subsubsection{Construction}

If $\Delta\ge\frac12$, the constant $\frac12$ has $\norm{F-\frac12}\le\frac12\le\Delta$ and $\dc=4+\lambda(1/2)=11\le11\Lambda$. Assume $\Delta<\frac12$ from now on. Put
\[
q:=\lceil4/\Delta\rceil,\qquad N:=\max\{1,\lceil qL\rceil\}\ (\ge qL),\qquad z_i:=i/N\ \ (i\in\{0,\dots,N\}^d),
\]
and $p_i:=\lfloor qF(z_i)\rceil\in\{0,\dots,q\}$, $y_i:=p_i/q$, so $|y_i-F(z_i)|\le\frac1{2q}$.

\paragraph{Step 1 (one-dimensional change of basis).} Let $\varphi^{(0)}(t)=1$ and $\varphi^{(a)}(t)=N\max\{0,t-\tfrac{a-1}N\}$ for $a=1,\dots,N$. For $g:\{0,\dots,N\}\to\R$ define
\[
(Dg)_0=g(0),\qquad (Dg)_1=g(1)-g(0),\qquad (Dg)_a=g(a)-2g(a-1)+g(a-2)\ \ (a\ge2).
\]
Then the piecewise linear interpolant of $g$ on the grid $\{i/N\}$ equals $\sum_{a=0}^N(Dg)_a\,\varphi^{(a)}(t)$ on $[0,1]$.

\emph{Proof.} The right side is continuous, equals $g(0)$ at $0$, and on $[\frac iN,\frac{i+1}N]$ has slope $N\bigl[(g(1)-g(0))+\sum_{a=2}^{i+1}(g(a)-2g(a-1)+g(a-2))\bigr]$, which telescopes to $N(g(i+1)-g(i))$. \qed

\paragraph{Step 2 (tensorization).} Let $\chi_i$ be the one-dimensional hat functions on the grid, and let
$M(x)=\sum_vy_v\prod_k\chi_{v_k}(x_k)$ be the multilinear interpolant of $(y_v)$. Then
\[
M(x)=\frac1q\sum_{a\in\{0,\dots,N\}^d}e_a\prod_{k=1}^d\varphi^{(a_k)}(x_k),\qquad e_a:=\sum_v\rho_{a,v}\,p_v,\quad \rho_{a,v}:=\prod_{k=1}^dD_{a_k,v_k},
\]
where $D_{a,v}$ is the matrix of the operator $D$.

\emph{Proof.} For product-form data $g=g_1\otimes\cdots\otimes g_d$, both sides equal $\prod_k(Pg_k)(x_k)$, where $P$ is the one-dimensional interpolation operator of Step~1. Product-form data span all $g$, and both sides are linear in $g$. \qed

\paragraph{Step 3 (bounded integer coefficients).} For $a\ne0$ we have $|e_a|\le E:=3\cdot4^d$, and $e_0=p_0$.

\emph{Proof.}
(i) The row sums of $D$ are $1$ for $a=0$ and $0$ for $a\ge1$. Hence $\sum_v\rho_{a,v}=\prod_k(\text{row sum at }a_k)=0$ whenever some $a_k\ge1$, i.e.\ whenever $a\ne0$.
(ii) The $\ell^1$-norms of the rows of $D$ are $1,2,4$ for $a=0$, $a=1$, $a\ge2$. Hence $\sum_v|\rho_{a,v}|\le4^d$.
(iii) The support of $\rho_{a,\cdot}$ lies in a box of index-width at most $2$ in each coordinate, so any two points $v,w$ in it satisfy $\lVert z_v-z_w\rVert_\infty\le2/N$.
(iv) For such $v,w$: $|p_v-p_w|=q|y_v-y_w|\le q\bigl(2L/N+1/q\bigr)\le2qL/N+1\le3$, using $N\ge qL$; the $p$'s are integers.
(v) By~(i), $e_a=\sum_v\rho_{a,v}(p_v-p_w)$ for any fixed $w$ in the support, so $|e_a|\le3\cdot4^d$ by~(ii) and~(iv). Finally $\rho_{0,v}=\delta_{v,0}$, so $e_0=p_0$. \qed

Hence for $a\ne0$ each $e_a$ fits in $b:=2d+3$ bits in offset-binary form, since $2^{b-1}=4^{d+1}\ge E+1$. Only $e_0=p_0$ is large, and it is a single number.

\paragraph{Step 4 (interpolation error).} The weights $\theta_v(x)=\prod_k\chi_{v_k}(x_k)$ are nonnegative, sum to $1$, and are supported on vertices of the cell containing $x$, which lie within $\ell_\infty$-distance $1/N$ of $x$. Hence
\[
|F(x)-M(x)|\le\sum_v\theta_v(x)\bigl(|F(x)-F(z_v)|+|F(z_v)-y_v|\bigr)\le\frac LN+\frac1{2q}\le\frac3{2q}\le\frac{3\Delta}8 ,
\]
using $L/N\le1/q$ and $1/q\le\Delta/4$.

\paragraph{Step 5 (smoothing).} Let $s:=\lceil\lg d\rceil+4d+2d\lceil\lg N\rceil$ and $\gamma:=2^{-s}$. Let $R_\gamma(z)=z/(1+e^{-z/\gamma})$, for which $|\max\{0,z\}-R_\gamma(z)|\le\gamma/3$ by Lemma~\ref{lem:ramp}. Put $\varphi_\gamma^{(0)}=1$, $\varphi_\gamma^{(a)}(t)=N\,R_\gamma\bigl(t-\tfrac{a-1}N\bigr)$ for $a=1,\dots,N$, and
\[
f(x):=\frac1q\sum_{a}e_a\prod_{k=1}^d\varphi_\gamma^{(a_k)}(x_k).
\]
On $[0,1]$ the exact factors satisfy $0\le\varphi^{(a)}\le N$, $|\varphi_\gamma^{(a)}-\varphi^{(a)}|\le N\gamma/3$, and the smoothed factors are at most $N(1+\gamma/3)\le2N$. Telescoping the product gives $\bigl|\prod\varphi^{(a_k)}-\prod\varphi_\gamma^{(a_k)}\bigr|\le d\,\frac{N\gamma}3(2N)^{d-1}$. There are at most $(N+1)^d\le(2N)^d$ terms and the term $a=0$ is exact. Therefore
\[
\norm{M-f}\le\frac Eq(2N)^d\cdot d\,2^{d-1}N^d\frac\gamma3=\frac{d\,2^{4d-1}N^{2d}\gamma}q\le\frac1{2q}\le\frac1q\le\frac\Delta4 ,
\]
because $\gamma\le1/(d\,2^{4d}N^{2d})$ by the choice of $s$. Together with Step~4, $\norm{F-f}\le\frac{3\Delta}8+\frac\Delta4<\Delta$.

\emph{Admissibility.} $R_\gamma(z)=z\div\bigl(1+\exp(0-z\times2^s)\bigr)$ has denominator $\ge1$, so $f$ is admissible and defined on all of $\R^d$.

\paragraph{Step 6 (formula and block structure).} The constant function $\varphi^{(0)}$ is not a ramp, so we split the sum over $a$ by $J:=\{k:a_k\ge1\}$. Then
\[
f=\frac1q\Bigl(p_0+\sum_{\emptyset\ne J\subseteq[d]}S_J\Bigr),\qquad S_J=\sum_{a\in\{1..N\}^J}e_{J,a}\prod_{k\in J}N\,R_\gamma\Bigl(x_k-\frac{a_k-1}N\Bigr),
\]
where $e_{J,a}$ denotes $e_a$ with $a_k=0$ for $k\notin J$. Each $S_J$ with $J=\{k_1<\dots<k_r\}$ is one $r$-fold block with $n=N$, $t_0=0$, $\tau=1/N$, $b=2d+3$, words $w_j=e_{J,j}+2^{b-1}$, and template
\[
T_J=\kappa\times\bigl(P_1\times(P_2\times\cdots\times P_r)\bigr),\qquad
P_m=c(N)\times\Bigl((x_{k_m}-\iota_m)\div\bigl(1+\exp\bigl(0-((x_{k_m}-\iota_m)\times c(2^s))\bigr)\bigr)\Bigr).
\]
Indeed $\iota_m$ becomes $(j_m-1)/N$, and $P_m=N\,R_\gamma(x_{k_m}-\iota_m)$ since $1/\gamma=2^s$. The whole formula is $\times\ c(1/q)\ \bigl(+\ c(p_0)\ (+\ S_{J_1}\ (\cdots))\bigr)$ with one $+$ per block.

\subsubsection{Bit count}\label{sec:bitsD}

Write $\ell:=\lg(N+1)$. We bound each part.
\begin{itemize}[itemsep=2pt]
\item \emph{Each $P_m$.} It has $16$ nodes ($64$ bits). Its payloads are: $c(N)$ costs $3+2\ell$; $c(1)$ costs $5$; $c(0)$ costs $3$; $c(2^s)$ costs $2s+3$; and the four leaves $x_{k_m},\iota_m,x_{k_m},\iota_m$ cost $4(2\lg d+1)$. Hence $|P_m|\le82+8\lg d+2\ell+2s$.
\item \emph{Each template.} The symbol $\kappa$ and $r$ symbols $\times$ add $4+4r$ bits, so $|T_J|\le r(86+8\lg d+2\ell+2s)+4$. Since $s\le\lg d+6d+1+2d\lg N$ and $\lg N\le\ell$, we have $2s\le2\lg d+12d+2+4d\ell$, and with $r\le d$,
\[
|T_J|\le 4+d(88+10\lg d+12d)+(4d^2+2d)\ell .
\]
\item \emph{Each block header.} It costs $4+\bigl(2\lg d+1\bigr)+\bigl(2\lg N+1\bigr)+3+\bigl(5+2\lg N\bigr)+\bigl(2\lg(2d+3)+1\bigr)=15+2\lg d+4\lg N+2\lg(2d+3)$ bits.
\item \emph{Block overhead (header plus template).} Adding the two bounds gives at most $u_d^{\circ}+v_d\ell$ with $u_d^{\circ}=19+2\lg d+2\lg(2d+3)+88d+10d\lg d+12d^2$ and $v_d=4d^2+2d+4$. Since $d-1-\lg d\ge0$ for every integer $d\ge1$, we have $u_d^{\circ}\le u_d:=19+2\lg d+2\lg(2d+3)+86d+8d\lg d+14d^2$. So the overhead is at most $u_d+v_d\ell$.
\item \emph{Coefficient words.} They total $(2d+3)\sum_{r\ge1}\binom dr N^r\le(2d+3)(N+1)^d\le(2d+3)2^dN^d$ bits.
\item \emph{Top level.} The symbols $\times$, $c$, $c$ and $2^d-1$ symbols $+$ cost $4\cdot2^d+8$ bits, and the constants $1/q$ and $p_0$ cost at most $8+4\lg(q+1)$ bits, using $p_0\le q$. In total at most $4\cdot2^d+16+4\lg(q+1)$.
\end{itemize}
There are $2^d-1$ blocks. Now $q+1\le\frac4\Delta+2\le(\frac1\Delta+2)^2$, so $\lg(q+1)\le2\Lambda$. Also $\ell\le N\le N^d$, and $\Lambda>2$ so any constant $c$ satisfies $c\le c\Lambda/2$. Summing,
\[
\dc(f)\le A'_dN^d+B'_d\Lambda,\qquad A'_d=2^d(4d^2+4d+7),\quad B'_d=8+\tfrac12\bigl(2^d(u_d+4)+16\bigr).
\]

\paragraph{From $N$ to $L/\Delta$.} If $qL\ge1$, then $N\le qL+1\le2qL<10L/\Delta$, because $q<5/\Delta$ for $\Delta<1$; so $N^d<10^dA$. If $qL<1$, then $N=1$ and $A'_d\le A'_d\Lambda/2$. In both cases $\dc(f)\le10^dA'_d\,A+(B'_d+A'_d/2)\Lambda$. This is the upper bound in Theorem~\ref{thm:mainD} with $\alpha_d=10^dA'_d=20^d(4d^2+4d+7)$ and $\beta_d=\max\{11,B'_d+A'_d/2\}$, the value $11$ covering $\Delta\ge\frac12$. \qed

\subsection{Remarks}

\begin{remark}[Where the dimension enters]
The lower bound counts $m^d$ disjoint bumps, so the entropy is of order $(L/\Delta)^d$, which matches the Kolmogorov--Tikhomirov estimate~\cite{KT1959}. The upper bound pays $O_d(1)$ bits per grid cell, thanks to the zero-row-sum argument of Step~3. The additive $\lg(1/\Delta)$ term does not depend on $d$.
\end{remark}

\begin{remark}[Constants]
The ratio of the upper to the lower constant is $2\cdot160^d(4d^2+4d+7)=2^{O(d)}d^2$. For growing $d$ the exponent $d$ is sharp but the constants are far from tight. None of them is optimized.
\end{remark}

\begin{remark}[Other norms]
Since $\lVert\cdot\rVert_\infty\le\lVert\cdot\rVert_2\le\sqrt d\,\lVert\cdot\rVert_\infty$, we have $\Lip^{\infty}_L\subseteq\Lip^{2}_L\subseteq\Lip^{\infty}_{L\sqrt d}$. For the Euclidean Lipschitz condition the lower bound therefore holds with $L$ unchanged, and the upper bound holds with $L$ replaced by $L\sqrt d$, which costs an extra factor $d^{d/2}$.
\end{remark}

\begin{remark}[Plain trees]
The upper bound uses the nested-loop block of $\mathcal L_d$. Without blocks, the same construction writes out the $(N+1)^d-1$ terms explicitly. Each term carries index constants and $2^s$ costing $O_d(\lg N)$ bits, so it gives $O_d\bigl((L/\Delta)^d\lg(L/\Delta+2)+\Lambda\bigr)$, and the plain-tree gap of Proposition~\ref{prop:plain} persists in every dimension. We do not claim that this cost is necessary for the construction or for the problem.
\end{remark}

\section{Numerical Verification}\label{sec:numerics}

The statements above are proved; this section checks, by running code, that the definitions are consistent, that the constructions behave as the proofs say, and that the scalings agree with the bounds. All code (Python, NumPy/SciPy) is in the supplementary archive and every table and figure below is produced by it with fixed random seeds. The numerical evidence supports the proofs but is not a substitute for them: the test functions are random or hand-built, not worst cases, and the constants of the theorems are very loose.

\subsection{What was implemented}

\begin{itemize}[itemsep=2pt]
\item A reference \emph{encoder}, an independently written \emph{decoder}, and an \emph{evaluator} for the description language of Definition~\ref{def:language} (with $d=1$) and for $\mathcal L_d$ of Section~\ref{sec:multid} (any $d$): symbol codes, Elias $\gamma$ codes, rational codes, loop blocks and templates.
\item The constructions of Theorem~\ref{thm:upper} ($d=1$) and Theorem~\ref{thm:mainD} ($d\ge1$), building the formula as a syntax tree, encoding it to a bit string, decoding it back, and evaluating the \emph{decoded} formula.
\item Test functions: sums of triangle waves with Lipschitz constant exactly $L$ (``tri''), and unions of saturating pyramids with slope exactly $L$ (``pyr'', the hard family of Lemma~\ref{lem:bumpD}), clipped to $[0,1]$.
\end{itemize}

\subsection{Encoder--decoder round trip}

Table~\ref{tab:rt} lists eight instances. In every case the decoded tree equals the encoded tree, the length of the bit string is the length of the description of the construction (hence an upper bound on its $\dc$), and the decoded formula agrees with the directly computed approximant to floating-point accuracy ($\le10^{-12}$). The resulting error is far below $\Delta$. The two length columns illustrate the effect of the loop block: for $d=1$, $\Delta=0.01$, $L=2$ the block language needs $1{,}878$ bits and the plain language $70{,}343$ bits for the same formula. Figure~\ref{fig:construction} shows one instance.

\begin{table}[t]
\centering\small
\resizebox{\linewidth}{!}{\begin{tabular}{rrrrrrrcrr}
\hline
$d$ & $\Delta$ & $L$ & $N$ & $q$ & bits (block) & bits (plain) & round trip & $\max|f_{\rm dec}-f|$ & $\|F-f\|_\infty/\Delta$ \\
\hline
1 & 0.2 & 3 & 30 & 10 & 328 & 3\,917 & yes & $4.4e-16$ & 0.185 \\
1 & 0.05 & 1 & 40 & 40 & 390 & 5\,657 & yes & $4.4e-16$ & 0.228 \\
1 & 0.01 & 2 & 400 & 200 & 1\,878 & 70\,343 & yes & $4.4e-16$ & 0.246 \\
2 & 0.3 & 1.5 & 21 & 14 & 4\,162 & 171\,553 & yes & $8.7e-14$ & 0.181 \\
2 & 0.4 & 1 & 10 & 10 & 1\,569 & 37\,436 & yes & $2.1e-14$ & 0.136 \\
2 & 0.45 & 2 & 18 & 9 & 3\,301 & 125\,654 & yes & $2.2e-13$ & 0.143 \\
3 & 0.45 & 1.2 & 11 & 9 & 17\,936 & 943\,929 & yes & $7.8e-13$ & 0.151 \\
3 & 0.49 & 1 & 9 & 9 & 11\,384 & 526\,631 & yes & $1.7e-13$ & 0.113 \\
\hline
\end{tabular}
}
\caption{Round trip encode $\to$ decode $\to$ evaluate. ``Bits (block)'' is the length of the encoded string in the language with loops, ``bits (plain)'' the length of the same formula written without loops.}\label{tab:rt}
\end{table}

\begin{figure}[t]
\centering\includegraphics[width=\linewidth]{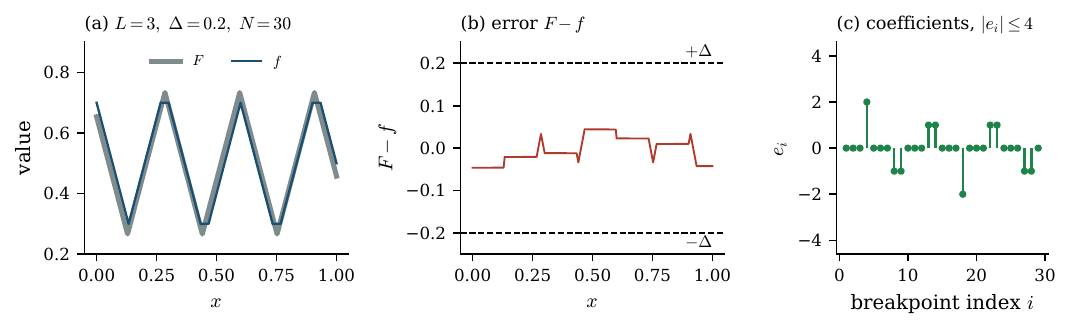}
\caption{The construction of Theorem~\ref{thm:upper} for $L=3$, $\Delta=0.2$ ($N=30$). (a) The function $F$ and the approximant $f$. (b) The uniform error stays inside $[-\Delta,\Delta]$. (c) The integer coefficients $e_i$ (second differences) take values in $\{-4,\dots,4\}$ and cost $O(1)$ bits each.}\label{fig:construction}
\end{figure}

\subsection{Accuracy and integer coefficients}

We built approximants for random instances: \nInstOne{} for $d=1$ ($\Delta\in[0.01,0.45]$, $L\in[0.5,8]$), \nInstTwo{} for $d=2$ and \nInstThree{} for $d=3$, alternating the two function types, and measured $\|F-f\|_\infty/\Delta$ on random points and on points next to the breakpoints. The largest observed ratios are \maxRatioOne{} ($d=1$), \maxRatioTwo{} ($d=2$) and \maxRatioThree{} ($d=3$), compared with the proved guarantees $2/3$, $5/8$ and $5/8$ (Figure~\ref{fig:err}). The empirical Lipschitz constants of the test functions never exceeded $L$.

The integer coefficients of Step~3 of Section~\ref{sec:upperD} stay far below the proved bound $3\cdot4^d$: the largest observed $|e_a|$ over 14 instances per dimension is \emaxOne{}, \emaxTwo{}, \emaxThree{}, \emaxFour{} for $d=1,\dots,4$, against \eboundOne{}, \eboundTwo{}, \eboundThree{}, \eboundFour{} (Figure~\ref{fig:coeffd}a). This is the property that keeps the cost per coefficient at $O(d)$ bits instead of $\lg(1/\Delta)$.

\begin{figure}[t]
\centering\includegraphics[width=\linewidth]{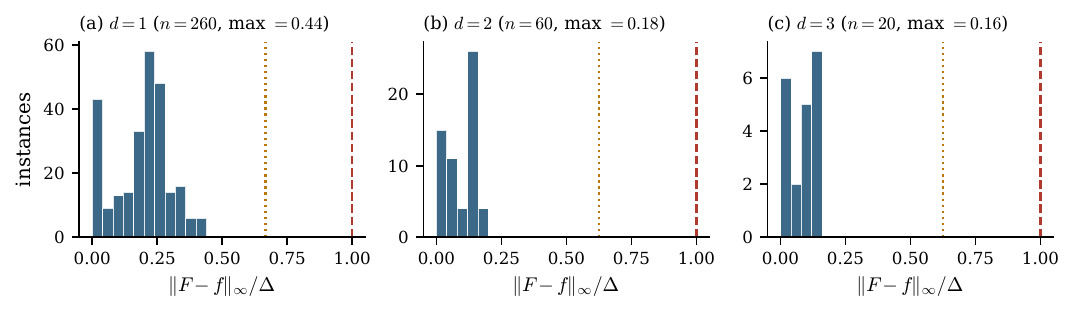}
\caption{Distribution of $\|F-f\|_\infty/\Delta$ over random instances. The dotted line is the proved worst-case ratio ($2/3$ for $d=1$, $5/8$ for $d\ge2$), the dashed line is $1$.}\label{fig:err}
\end{figure}

\subsection{Scaling of the description length}

Figure~\ref{fig:scale1} (a) plots the measured length against $L/\Delta$ for $d=1$ up to $L/\Delta\approx\maxRatioL{}$. With the loop block the length grows essentially linearly: the fitted log--log slope for $L/\Delta\ge50$ is \slopeOne{} (it tends to $1$ as the additive terms become negligible), and the length per unit of $L/\Delta$ settles near \bitsPerRatioMin{} bits, in line with $4N$ coefficient bits and $N\approx2L/\Delta$. The measured length never exceeded $\maxUbFrac$ of the proved upper bound $435\,L/\Delta+12\Lambda$, and it was at least \minLbRatio{} times the lower-bound expression $(L/\Delta+\Lambda)/16$; the latter bound is a worst-case statement over the class, so a single explicit approximant exceeding it is consistent. Without loops the same formulas have fitted slope \slopePlainOne{} and the length per unit of $L/\Delta$ keeps growing (to \plainPerRatioMax{} bits), the extra logarithmic factor of Proposition~\ref{prop:plain} (Figure~\ref{fig:scale1} (b)).

For $d=2,3$ the measured length grows roughly linearly in $A=(L/\Delta)^d$ (slopes below $1$ because the additive terms are still visible at these moderate sizes): the fitted slopes are \slopeTwo{} ($d=2$, up to $A\approx\maxATwo{}$) and \slopeThree{} ($d=3$, up to $A\approx\maxAThree{}$), and the length per unit of $A$ at the largest instance is \bitsPerATwo{} and \bitsPerAThree{} bits respectively (Figure~\ref{fig:coeffd}b); the growth of this constant with $d$ reflects the $2^d-1$ blocks and the $2d+3$ bits per coefficient, i.e.\ the exponential dependence of the constants on $d$.

\begin{figure}[t]
\centering\includegraphics[width=\linewidth]{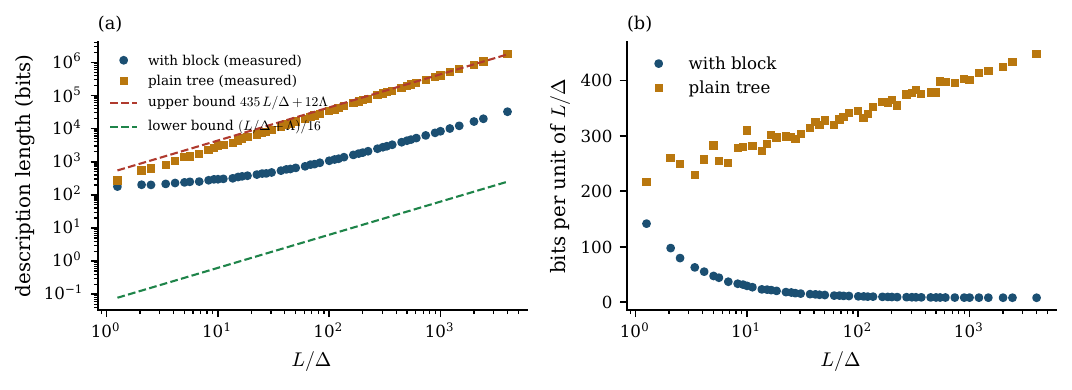}
\caption{Description length of the construction for $d=1$ (four values of $L$, thirteen values of $\Delta$ each). (a) Log--log plot against $L/\Delta$ with the proved bounds. (b) Bits per unit of $L/\Delta$: flat with the loop block, growing without it.}\label{fig:scale1}
\end{figure}

\begin{figure}[t]
\centering
\begin{minipage}[b]{0.43\linewidth}\centering
\includegraphics[width=\linewidth]{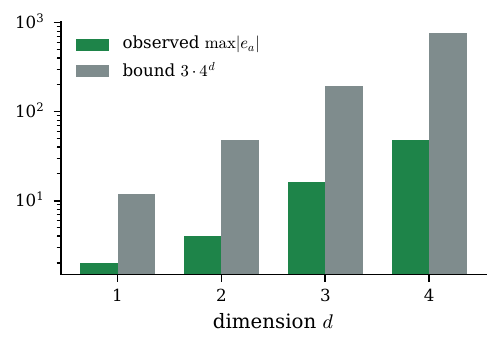}\\[-2pt](a)
\end{minipage}\hfill
\begin{minipage}[b]{0.55\linewidth}\centering
\includegraphics[width=\linewidth]{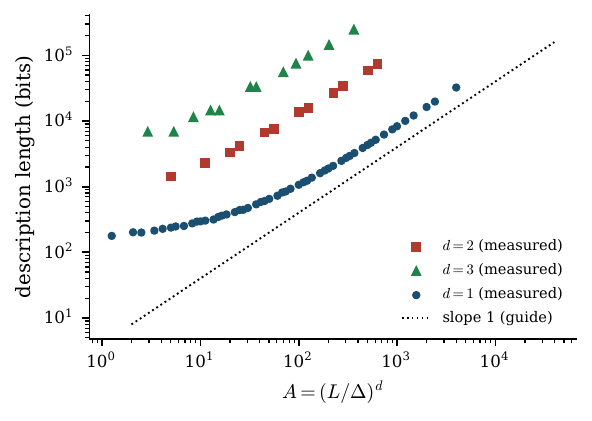}\\[-2pt](b)
\end{minipage}
\caption{(a) Largest observed integer coefficient against the bound $3\cdot4^d$. (b) Measured description length against $A=(L/\Delta)^d$ for $d=1,2,3$.}\label{fig:coeffd}
\end{figure}

\subsection{Packing family and the ramp lemma}

For the lower bound we generated the bump families of Lemma~\ref{lem:bumpD} for five choices of $(d,L,\Delta)$ and checked, on random pairs of members, that the sup-distance equals $\eta>2\Delta$ and that the empirical Lipschitz ratio does not exceed $L$ (Table~\ref{tab:pack}, Figure~\ref{fig:pack}). For the ramp of Lemma~\ref{lem:ramp}, the maximum of $\varphi(v)=v/(1+e^{v})$ is \rampMax{} at $v\approx\rampArg$, below the bound $1/3$, and over $2\times10^5$ random pairs $(z,\gamma)$ the largest observed value of $|\max\{0,z\}-R_\gamma(z)|/\gamma$ is \rampWorst{} (Figure~\ref{fig:ramp}).

\begin{table}[t]
\centering\small
\resizebox{\linewidth}{!}{\begin{tabular}{rrrrrrrrr}
\hline
$d$ & $L$ & $\Delta$ & $m$ & $\lg|\mathcal G|=m^d$ & $\eta$ & $2\Delta$ & min.\ distance & max.\ Lip.\ ratio \\
\hline
1 & 6 & 0.1 & 14 & 14 & 0.214 & 0.200 & 0.214 & 6.000 \\
1 & 20 & 0.2 & 24 & 24 & 0.417 & 0.400 & 0.417 & 20.000 \\
2 & 5 & 0.2 & 6 & 36 & 0.417 & 0.400 & 0.417 & 5.000 \\
2 & 8 & 0.3 & 6 & 36 & 0.667 & 0.600 & 0.667 & 8.000 \\
3 & 6 & 0.3 & 4 & 64 & 0.750 & 0.600 & 0.750 & 6.000 \\
\hline
\end{tabular}
}
\caption{Check of the packing family of Lemma~\ref{lem:bumpD}: $m^d$ bits are forced because the members are more than $2\Delta$ apart.}\label{tab:pack}
\end{table}

\begin{figure}[t]
\centering\includegraphics[width=\linewidth]{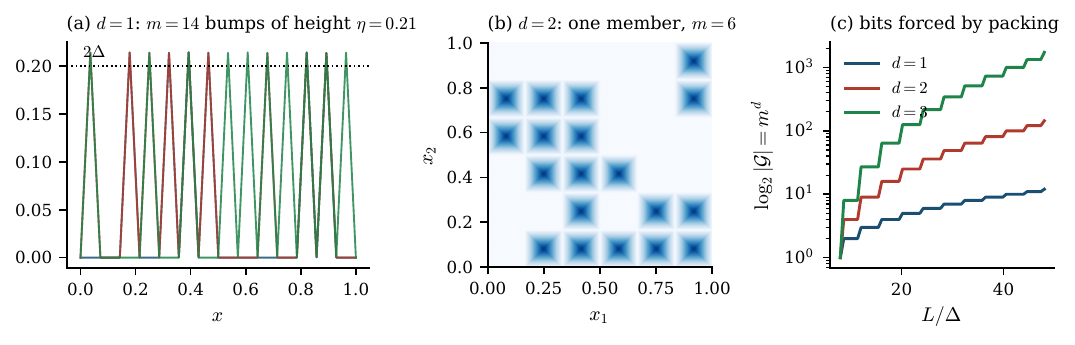}
\caption{The hard family of Lemma~\ref{lem:bumpD}. (a) Three members for $d=1$. (b) One member for $d=2$. (c) The number $m^d$ of independent bits against $L/\Delta$.}\label{fig:pack}
\end{figure}

\begin{figure}[t]
\centering\includegraphics[width=0.85\linewidth]{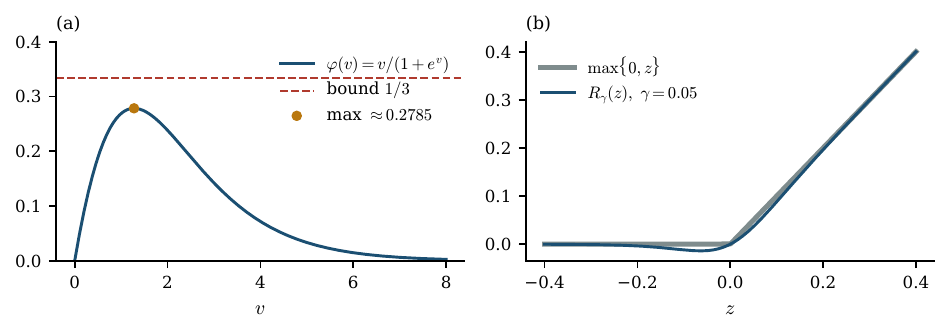}
\caption{(a) The function $\varphi$ of Lemma~\ref{lem:ramp} and its maximum. (b) The smooth ramp $R_\gamma$ against $\max\{0,z\}$.}\label{fig:ramp}
\end{figure}
\section{Discussion}\label{sec:discussion}

\subsection{What depends on the description language}

The same holds in every dimension (Section~\ref{sec:multid}).\par

The lower bounds (Theorem~\ref{thm:lower}, Proposition~\ref{prop:lowerlog}) hold for every description language, so $\Omega(L/\Delta+\lg(1/\Delta+2))$ is a property of the pair (Lipschitz class, uniform error) rather than of our encoding. The upper bound of Theorem~\ref{thm:upper} uses the repetition construct (D5). For plain trees, Proposition~\ref{prop:plain} gives only
\[
\frac1{16}\Bigl(\frac L\Delta+\Lambda_\Delta\Bigr)\le\CC_{\mathrm{plain}}(L,\Delta)\le O\Bigl(\frac L\Delta\lg\Bigl(\frac L\Delta+2\Bigr)+\Lambda_\Delta\Bigr)
\]
for $\Delta\le1/36$. A plain tree cannot share the information in the regular breakpoints $i/N$, and it is open whether other constructions avoid the extra factor.

\subsection{Programmatic complexity}

One may ask instead how long a \emph{program} must be to print a good formula. Fix a universal Turing machine, and let $\dc_{\mathrm{prog}}(F,\Delta)$ be the length of the shortest stand-alone program (with no access to $F$) that halts and prints a description of an admissible $f$ with $\norm{F-f}\le\Delta$.

\begin{proposition}[Worst case]\label{prop:prog}
For $0<\Delta\le1/36$ and $L>0$,
\[
\frac1{16}\Bigl(\frac L\Delta+\Lambda_\Delta\Bigr)\ \le\ \sup_{F\in\Lip_L}\dc_{\mathrm{prog}}(F,\Delta)\ \le\ 435\,\frac L\Delta+12\,\Lambda_\Delta+c_U,
\]
where $c_U$ depends only on the machine.
\end{proposition}
\begin{proof}
At most $2^{M+1}-1$ programs have length $\le M$, and each prints at most one description, hence at most one formula. Lemma~\ref{lem:pigeon} and the proofs of Theorem~\ref{thm:lower}, Proposition~\ref{prop:lowerlog} and Corollary~\ref{cor:theta} use only that fewer than $2^{M+1}$ candidate formulas have cost $\le M$, so they go through with ``program'' in place of ``description''. For the upper bound, a program that prints a fixed description $w$ has length $|w|+c_U$.
\end{proof}

A program that \emph{computes} its constants rather than listing them therefore gains nothing in the worst case: there are too few programs. Gains are possible only for individual functions. Two relative variants are of interest.
\begin{itemize}[itemsep=2pt]
\item \emph{Description model.} If the program is given a machine $M_F$ computing $F$, the loop of Step~7 yields length $K(M_F)+O(\lg(L/\Delta+2)+\Lambda_\Delta)$ (using finite-precision evaluation with $q=\lceil4/\Delta\rceil$, which only changes constants). This is small only for functions of low Kolmogorov complexity.
\item \emph{Oracle model.} If the program may query $(x,q)\mapsto\lfloor qF(x)\rceil$, its length is $O(\lg(L/\Delta+2)+\Lambda_\Delta)$ independent of $F$. Counting does not apply here, because the oracle supplies unboundedly many bits, so this quantity does not measure the complexity of $F$.
\end{itemize}

\subsection{Limitations}

\begin{itemize}[itemsep=2pt]
\item The main sections treat one variable; Section~\ref{sec:multid} extends the Lipschitz results to $d$ variables with $d$-dependent constants. Other function classes are not treated. For values in $[0,B]$, apply the results to $F/B$ with $\Delta/B$ and $L/B$, and rescale the output.
\item The constants ($435$, $12$, $1/16$, $1/36$) are not optimized. Only the orders of magnitude are claimed.
\item The upper bound is existential given the values $F(i/N)$. We do not analyse the cost of computing the description from $F$.
\item $\dc$ is relative to the language of Definition~\ref{def:language}. The lower bounds are language independent, the upper bounds are not (Remark~\ref{rem:language}).
\item Error is measured only for admissible formulas (Definition~\ref{def:formula}). The lower-bound argument counts all formulas, so it is insensitive to this convention.
\item The hard functions of Lemma~\ref{lem:bump} are explicit but not natural. Typical-case behaviour is not addressed.
\end{itemize}

\subsection{Open problems}\label{sec:open}

\begin{enumerate}[itemsep=2pt]
\item \emph{Plain trees.} Is $\CC_{\mathrm{plain}}(L,\Delta)=\Theta(L/\Delta)$ for $L\ge1$, or does a lower bound of order $(L/\Delta)\lg(L/\Delta)$ hold?
\item \emph{Sharp constants} in Theorems~\ref{thm:upper} and~\ref{thm:lower}.
\item \emph{Node complexity.} What are the bounds in terms of $|f|$ when constants are free? Pseudo-dimension bounds for Pfaffian formulas~\cite{KM1997} are a natural tool.
\item \emph{Instance-wise program length.} The worst case matches (Proposition~\ref{prop:prog}). For an individual $F$, how does $\dc_{\mathrm{prog}}(F,\Delta)$ relate to the Kolmogorov complexity of $F$ and to $\Delta$?
\item \emph{Smoothness classes.} For H\"older or Sobolev balls of smoothness $r$ in dimension $d$ the entropy is of order $(L/\Delta)^{d/r}$~\cite{KT1959}. We expect $\Theta(\Delta^{-d/r})$ bits with local Taylor polynomials and a smooth partition of unity, but have not proved it. Bounded-variation classes, other bases and restricted bases are also open.
\item \emph{Average case.} The typical complexity of a random Lipschitz function.
\end{enumerate}

\subsection{Relation to approximation by neural networks}

Yarotsky and Zhevnerchuk~\cite{YZ2020} study how the error of ReLU and periodic-activation networks scales with the number of weights. Their accounting of information (their Table~1) is that the total information needed to approximate a smoothness-$r$ ball in dimension $d$ to accuracy $\epsilon$ is always of order $\epsilon^{-d/r}$, and that networks differ only in how this is split between the number of weights and the number of bits per weight. The description complexity $\dc$ charges exactly this total, so the lower bounds of Section~\ref{sec:lower} and Section~\ref{sec:multid} are the Boolean-complexity counterpart of that entropy count. The superexpressive activations of~\cite{Y2021} exploit unbounded precision of two real weights; the bit-counting measure $\dc$ rules this out by construction. The basis $\{+,-,\times,\div,\exp\}$ consists of Pfaffian functions, so Khovanskii-type bounds~\cite{Khovanskii1991,KM1997}, as used in~\cite{Y2021} to show that standard activations are not superexpressive, would be the tool for the open problem on node complexity.

\section{AI Disclosure Statement}\label{sec:ai}

This statement follows the Leiden Declaration~\cite{Leiden} (2 June 2026, endorsed by the International Mathematical Union). Full responsibility for the correctness of all results lies with the author. No AI system is an author of this paper.

\paragraph{Role of AI (large language model: Claude, Anthropic).}
\begin{itemize}[itemsep=2pt]
\item Typesetting and academic formatting of the manuscript, and layout of the presentation slides.
\item Simulation of peer review and adversarial questioning of earlier drafts.
\item Suggestions for the technical presentation of the proofs, including rewriting of statements and proofs in response to the simulated review (in particular the corrected lower-bound argument and the precise definition of description complexity).
\item Assistance in drafting the text of the multidimensional extension (Section~\ref{sec:multid}) and in formalizing its tensor-product construction; the extension idea and its mathematical validation are the author's.
\item Writing the code for the numerical verification and producing the figures and tables of Section~\ref{sec:numerics}.
\end{itemize}

\paragraph{Contribution of the author.}
\begin{itemize}[itemsep=2pt]
\item Derivation of the key mathematical formulas (polished with the help of AI).
\item Manual verification of the rigorous proofs.
\item Development of the original theoretical concept, including the multidimensional extension.
\item Final approval of all results.
\end{itemize}

\section{Conclusion}

For all $L>0$ and $\Delta\le1/36$, the worst-case number of bits needed to approximate an $L$-Lipschitz function on $[0,1]$ to uniform error $\Delta$ by an elementary expression over $\{+,-,\times,\div,\exp,x\}$ with rational constants is $\Theta\bigl(L/\Delta+\lg(1/\Delta+2)\bigr)$, and on $[0,1]^d$ it is $\Theta_d\bigl((L/\Delta)^d+\lg(1/\Delta+2)\bigr)$. The upper bounds use a smoothed spline in a description language with repetition constructs, and the lower bounds are counting arguments valid in every description language. A reference implementation confirms the constructions and the predicted scalings. For plain expression trees without repetition, the upper bounds carry an extra factor $\lg(L/\Delta)$, and closing that gap is the main open question.

\end{document}